\documentclass[10pt]{article}

\usepackage{amsmath,amsthm,amssymb}
\usepackage[bookmarksnumbered=true,pagebackref]{hyperref}
\usepackage{tocbibind}
\usepackage{enumitem}

\newtheorem{thm}{Theorem}
\newtheorem{prop}{Proposition}
\newtheorem{lem}{Lemma}

\theoremstyle{remark}
\newtheorem{rem}{Remark}
\newtheorem*{acknowledgments}{Acknowledgments}

\theoremstyle{definition}

\newtheoremstyle{notes}% name
{3pt}% Space above
{3pt}% Space below
{}% Body font
{}% Indent amount
{\bfseries}% Theorem head font
{:}% Punctuation after theorem head
{.4em}% Space after theorem head
{}% Theorem head spec (can be left empty, meaning ‘normal’)
\theoremstyle{notes}
\newtheorem*{keywords}{Keywords}
\newtheorem*{subjclass}{AMS MSC 2010}
\newtheorem{example}{Example}

\DeclareMathOperator*{\Var}{Var}
\newcommand{\D}[1]{\mathop{\mathrm{d}#1}}

\usepackage[draft]{todonotes} 
\usepackage{comment}
\title{CLT with explicit variance for products of random singular matrices related to Hill's equation\footnote{Work supported in part by NSF Grant DMS-1659643.}}

\author{Phanuel Mariano\footnote{Supported at Union College in part by an AMS-Simons Travel Grant 2019-2022.}\\
\href{mailto:marianop@union.edu}{\texttt{{\small marianop@union.edu}}
}
\and
Hugo Panzo\footnote{Supported at the Technion by a Zuckerman Fellowship.}\\ 
\href{mailto:panzo@campus.technion.ac.il}{\texttt{{\small panzo@campus.technion.ac.il}}}}
\date{\today}

\begin{document}

\maketitle

\begin{abstract}
We prove a central limit theorem (CLT) for the product of a class of random singular matrices related to a random Hill's equation studied by Adams--Bloch--Lagarias. The CLT features an explicit formula for the variance in terms of the distribution of the matrix entries and this allows for exact calculation in some examples. Our proof relies on a novel connection to the theory of $m$-dependent sequences which also leads to an interesting and precise nondegeneracy condition.
\end{abstract}

\begin{keywords}
Lyapunov exponent; products of random matrices; central limit theorem; Hill's equation; $m$-dependent sequences. 
\end{keywords}

\begin{subjclass}
Primary 37H15, 60B20, 60F05; Secondary 34F05, 60G10.
\end{subjclass}

\section{Introduction}

\subsection{Random Hill's equation}\label{sec:random_Hills}

In this paper we consider the product of a sequence $\{Y_j\}_{j\geq 1}$ of random singular matrices of the form
\begin{equation}\label{eq:1}
Y_j=\left[\begin{array}{cc}
1 & x_j\\
\frac{1}{x_j} & 1
\end{array}\right]
\end{equation}
where $\{x_j\}_{j\geq 1}$ is an $\mathrm{i.i.d.}$ sequence of $\mathbb{R}\setminus\{0\}$-valued random variables. These random matrices are related to a random Hill's equation studied by Adams, Bloch, and Lagarias in \cite{Adams-Bloch2008,Adams-Bloch2009,Adams-Bloch2010,Hills_2013,Hills}. This equation is a generalization of the classical deterministic Hill's equation \cite{Hill-1986} which originally modeled lunar orbits and has more recently been applied to the study of many other physical systems. See the monograph \cite{Hills_book} in addition to the more recent articles cited above and references therein for further results on the classical Hill's equation and its many applications. 

The random Hill's equation studied by Adams--Bloch--Lagarias that is related to our work is the second order $\pi$-periodic ODE 
\begin{equation}\label{eq:Hills_ODE}
\frac{\D{}^2 y_j}{\D{t}^2}+\left(\lambda_j+q_j\,\hat{Q}(t)\right)y_j=0,~(j-1)\,\pi\leq t\leq j\,\pi
\end{equation}
where the random real-valued oscillation parameter $\lambda_j$ and forcing parameter $q_j$ form an form an $\mathrm{i.i.d.}$ sequence $\{(\lambda_j,q_j)\}_{j\geq 1}$ and $\hat{Q}$ is a periodic function with minimal period $\pi$ that is symmetric about $\frac{\pi}{2}$ with normalization $\int_0^\pi\hat{Q}(t)\D{t}=1$. 

In the remainder of this section, we draw from \cite{Adams-Bloch2010,Hills} to illustrate how equation \eqref{eq:Hills_ODE} is connected to the matrices \eqref{eq:1}. Towards this end, first consider any two successive cycles and their parameter realizations $(\lambda_j,q_j)$ and $(\lambda_{j+1},q_{j+1})$. The general solutions $y_j$ and $y_{j+1}$ of \eqref{eq:Hills_ODE} can be written as linear combinations of the normalized principal solutions $y_{1,j},\,y_{2,j}$ and $y_{1,j+1},\,y_{2,j+1}$. That is,
\begin{equation}\label{eq:principal}
\begin{split}
y_j(t) & =\alpha_j\,y_{1,j}(t)+\beta_j\,y_{2,j}(t),~(j-1)\,\pi\leq t\leq j\,\pi,\\
y_{j+1}(t) & =\alpha_{j+1}\,y_{1,j+1}(t)+\beta_{j+1}\,y_{2,j+1}(t),~j\,\pi\leq t\leq (j+1)\,\pi
\end{split}
\end{equation}
with the normalization
\begin{equation}\label{eq:boundary}
\begin{split}
y_{1,j}\big((j-1)\,\pi\big)&=y_{2,j}'\big((j-1)\,\pi\big)=y_{1,j+1}(j\,\pi)=y_{2,j+1}'(j\,\pi)=1,\\
y_{1,j}'\big((j-1)\,\pi\big)&=y_{2,j}\big((j-1)\,\pi\big)=y_{1,j+1}'(j\,\pi)=y_{2,j+1}(j\,\pi)=0.
\end{split}
\end{equation}

Since we want the concatenation of $y_j$ and $y_{j+1}$ to be continuously differentiable, it follows that $y_j(\pi\, j)=y_{j+1}(\pi\, j)$ and $y_j'(\pi\, j)=y_{j+1}'(\pi\, j)$. Together with \eqref{eq:principal} and \eqref{eq:boundary}, this implies that
\begin{equation}\label{eq:transform}
\begin{split}
\alpha_{j+1}&=\alpha_j\, y_{1,j}(j\,\pi)+\beta_j\, y_{2,j}(j\,\pi),\\
\beta_{j+1}&=\alpha_j \,y_{1,j}'(j\,\pi)+\beta_j\, y_{2,j}'(j\,\pi).
\end{split}
\end{equation}
The transformation \eqref{eq:transform} from the coefficients $(\alpha_j,\beta_j)$ of one cycle to the coefficients $(\alpha_{j+1},\beta_{j+1})$ of the next cycle can be represented as a matrix. Since the Wronskian of \eqref{eq:Hills_ODE} is $1$, the determinant of this transformation matrix must also be $1$. Moreover, $y_{1,j}(j\,\pi)=y_{2,j}'(j\,\pi)$ by the symmetry condition imposed on $\hat{Q}$. These facts allow us to write the transformation \eqref{eq:transform} in the form 
\[
\left[\begin{array}{c}
\alpha_{j+1}\\
\beta_{j+1}
\end{array}\right]=\underbrace{\left[\begin{array}{cc}
h_j & \left(h_j^2-1\right)/g_j\\
g_j & h_j
\end{array}\right]}_{\displaystyle M_j}\left[\begin{array}{c}
\alpha_j\\
\beta_j
\end{array}\right]
\]
where $h_j=y_{1,j}(j\,\pi)$ and $g_j=y_{1,j}'(j\,\pi)$. 

Next we note that since the parameter sequence is $\mathrm{i.i.d.}$, the pairs of normalized principal solutions will be also be $\mathrm{i.i.d.}$ across cycles. In light of \eqref{eq:principal}, this implies that any asymptotic growth or decay of the concatenated solutions $\{y_j\}_{j\geq 1}$ on $[0,\infty)$ is determined by the asymptotic behavior of the coefficients $(\alpha_j,\beta_j)$ as $j\to\infty$. In many cases it is the exponential rate of growth of the solutions that is of interest and this rate can often be deduced from the almost sure limit of $\frac{1}{n}\log\|M_n\cdots M_2\cdot M_1\|$ as $n\to\infty$ for any matrix norm $\|\cdot\|$.

As noted above, the pairs of normalized principal solutions are $\mathrm{i.i.d.}$ across cycles so consequently $\{h_j/g_j\}_{j\geq 1}$ is an $\mathrm{i.i.d.}$ sequence. Now we can substitute $x_j=h_j/g_j$ into the above formula for the transformation matrix $M_j$ to get an expression which involves matrices of the form \eqref{eq:1}, that is
\begin{equation}\label{eq:transform_mat}
M_j
=h_j\left[\begin{array}{cc}
1 & x_j\\
\frac{1}{x_j} & 1
\end{array}\right]+\left[\begin{array}{cc}
0 & -\frac{1}{g_j}\\
0 & 0
\end{array}\right].
\end{equation}
Letting $\|\cdot\|$ denote the Hilbert-Schmidt norm, we can compute the relative magnitudes of the matrix components on the right-hand side of \eqref{eq:transform_mat}, namely
\begin{equation}\label{eq:magnitudes}
\left\|\left[\begin{array}{cc}
1 & x_j\\
\frac{1}{x_j} & 1
\end{array}\right]\right\|=\left|\frac{g_j}{h_j}\right|+\left|\frac{h_j}{g_j}\right|~\text{ and }~\left\|\left[\begin{array}{cc}
0 & -\frac{1}{g_j}\\
0 & 0
\end{array}\right]\right\|=\frac{1}{|g_j|}.
\end{equation}

One scenario that is of particular interest is the so-called \emph{unstable regime} where $|h_j|\gg 1$. For instance, as remarked upon in Section 3 of \cite{Adams-Bloch2010}, this arises when the random forcing parameters satisfy $q_j\gg 1$. In this case it follows from \eqref{eq:magnitudes} that
\[
\left\|\left[\begin{array}{cc}
1 & x_j\\
\frac{1}{x_j} & 1
\end{array}\right]\right\|\gg\left\|\left[\begin{array}{cc}
0 & -\frac{1}{g_j}\\
0 & 0
\end{array}\right]\right\|
\]
and this suggests that we approximate \eqref{eq:transform_mat} with
\begin{equation}\label{eq:approx}
M_j\approx h_j\left[\begin{array}{cc}
1 & x_j\\
\frac{1}{x_j} & 1
\end{array}\right].
\end{equation}
As $\{h_j\}_{j\geq 1}$ is just an $\mathrm{i.i.d.}$ sequence of scalars, with this approximation the main task in computing the almost sure limit of $\frac{1}{n}\log\|M_n\cdots M_2\cdot M_1\|$ as $n\to\infty$ is reduced to computing the analogous limit for matrices of the form \eqref{eq:1}.

\subsection{Lyapunov exponent for products of random matrices}\label{sec:Lyapunov}

Let $\{Y_j\}_{j\geq 1}$ be an $\mathrm{i.i.d.}$ sequence of random $d\times d$ matrices, not necessarily of the form \eqref{eq:1}, and let $S_n=Y_n\cdots Y_2\cdot Y_1$ denote their partial products. We can define the growth rate or \emph{Lyapunov exponent} of the product by
\begin{equation}\label{eq:lyapunov_def}
\lambda :=\lim_{n\to\infty}\frac{1}{n}\mathbb{E}\big[\log\Vert S_{n}\Vert\big].
\end{equation}
That this limit exists in $[-\infty,\infty]$ follows from \cite[Theorem 1]{FurKest} and the fact that $\{Y_j\}_{j\geq 1}$ is an $\mathrm{i.i.d.}$ sequence. Note also that the Lyapunov exponent is invariant under the choice of matrix norm $\left\Vert \cdot\right\Vert $. 

Denote the positive part of $\log x$ by $\log^+ x$, that is, $\log^+ x=\max\{\log x,0\}$. Then assuming 
\[
\mathbb{E}\left[\log^+ \left\Vert Y_1\right\Vert\right]<\infty,
\]
Furstenberg and Kesten \cite[Theorem 2]{FurKest} showed that \eqref{eq:lyapunov_def} can be strengthened to an almost sure statement. More precisely, with probability $1$ we have
\[
\lambda=\lim_{n\to\infty}\frac{1}{n}\log\left\| S_{n}\right\|<\infty.
\]
This is often seen as an analogue of the classical strong law of large numbers. As mentioned in Section \ref{sec:random_Hills}, when the matrices are of the form \eqref{eq:1}, then this almost sure limit is connected to the growth rate of the solutions of the random Hills equation \eqref{eq:Hills_ODE} in the unstable regime where the approximation \eqref{eq:approx} is valid. 
 
Explicit formulas are a major focus of this paper and numerous authors have obtained such expressions for the Lyapunov exponent of specific $d\times d$ matrix models; see the papers \cite{Newman1,Mannion,Lima-Rahibe1994,Marklof-etall2008,Forrester2013,Kargin,Forrester-Zhang-2020} and monographs \cite{Bougerol,Crisanti,Finch} for some examples. However, in most cases the Lyapunov exponent cannot be computed explicitly and only estimates or an algorithm for its approximation can be given; see \cite{Tsitsiklis-Blondel1997,Viswanath2000,Pollicott2010,Protasov-Jungers-2013,Sturman-Thiffeault-2019,Rajeshwari-etall-2020,Lemm-Sutter-2020}.  

Most of the literature on products of random matrices deals with products in $\mbox{GL}(d,\mathbb{R})$ and requires assumptions on the subgroup generated by the support of the probability measure from which the random matrices are drawn. However, in the setting of \eqref{eq:1}, we lose these results since the matrices are all singular. For this reason, the authors of \cite{Hills} had to employ an ad hoc argument when proving the following explicit formula for the Lyapunov exponent of the product of random matrices drawn from this class. Generalizing this result in Theorem \ref{thm:lyapunov_new} is one of the main contributions of our paper.

\begin{prop}\label{prop:ABL_Lyapunov}
\emph{Adams--Bloch--Lagarias \cite[Theorem 17.2]{Hills}}\\
Consider the random singular matrices of the form 
\[
Y_j=\left[\begin{array}{cc}
1 & \epsilon_j\, r_j\\
\frac{1}{\epsilon_j\, r_j} & 1
\end{array}\right],~j\geq 1
\]
where $\{\epsilon_j\}_{j\geq 1}$ is an i.i.d. sequence of random signs with 
\[
\mathbb{P}(\epsilon_1=1)=p=1-\mathbb{P}(\epsilon_1=-1)
\]
and $\{r_j\}_{j\geq 1}$ is an i.i.d. sequence of positive random variables independent of the random signs such that
\[
\mathbb{E}\left[\log\left(1+r_1\right)\right]<\infty~\text{ and }~\mathbb{E}\left[\log\left(1+\frac{1}{r_1}\right)\right]<\infty.
\]
Then almost surely
\[
\lambda=\lim_{n\to\infty}\frac{1}{n}\log\left\Vert S_n\right\Vert<\infty
\]
and the Lyapunov exponent has the explicit representation
\begin{equation}\label{eq:lyapunov_old}
\lambda=\left(p^2+(1-p)^2\right)\mathbb{E}\left[\log\left(1+\frac{r_2}{r_1}\right)\right]+2p\left(1-p\right)\mathbb{E}\Bigg[\log\left|1-\frac{r_2}{r_1}\right|\Bigg].
\end{equation}
\end{prop}
\begin{rem}\label{rem:ABL_implicit}
The statement of Proposition \ref{prop:ABL_Lyapunov} used here is different than what appears in \cite[Theorem 17.2]{Hills} since we want to make explicit the sign independence condition which is implicitly assumed in that article. For instance, it is not clear from \cite{Hills} that their formula (17.23) \emph{cannot} be applied in cases such as Example \ref{ex:uniform} when $a>0$ and $a\neq b$. That is, when the uniform random variable takes both signs with the interval not being symmetric about $0$.
\end{rem}

\subsection{CLT for products of random matrices}\label{sec:variance}

There is also a central limit theorem (CLT) for the product of $\mbox{GL}(d,\mathbb{R})$-valued random matrices due to Tutubalin \cite{Tutubalin}, Le Page \cite{Le_Page}, and others; see \cite[Theorem V.5.4]{Bougerol} and also \cite{GL_CLT} for a recent version with the optimal moment condition. This limit theorem provides reasonable conditions on the distribution of the matrices which guarantee that
\[
\frac{1}{\sqrt{n}}\big(\log\left\| S_n\right\| -n\lambda\big)\stackrel{\mathcal{L}}{\to}N(0,\sigma^2)
\]
where by $\stackrel{\mathcal{L}}{\to}$ we mean convergence in law and $N(0,\sigma^2)$ denotes a Gaussian random variable with mean $0$ and variance $\sigma^2$. Following the discussion in Section \ref{sec:random_Hills}, we note that a comparable CLT for the random singular matrices of the form \eqref{eq:1} would have significance for describing the fluctuations of the solutions of the random Hills equation \eqref{eq:Hills_ODE} in the unstable regime where the approximation \eqref{eq:approx} is valid. 

Compared to $\lambda$, much less is known about $\sigma^2$. A perturbative approach was used in \cite{perturbative} to compute the leading order term in an asymptotic expansion for the variance of the product of random matrices which are close to the identity. In \cite{Rajeshwari-etall-2020}, Monte Carlo simulation was used to approximate the variance in a specific $2\times 2$ matrix model where $\lambda$ is known exactly. As for explicit formulas, in the special case where the distribution of $V_\mathbf{x}:=\|Y_1\mathbf{x}\|/\|\mathbf{x}\|$ does not depend on $\mathbf{x}\in\mathbb{R}^d\setminus\{\mathbf{0}\}$, Cohen and Newman \cite[Proposition 2.1]{newman} proved that
\begin{equation}\label{eq:newman}
\lambda=\mathbb{E}\left[\log V_\mathbf{x}\right]~\text{ and }~\sigma^2=\mathbb{E}\left[\left(\log V_\mathbf{x}-\lambda\right)^2\right]
\end{equation}
hold whenever $\mathbb{E}[(\log V_\mathbf{x})^2]<\infty$. In particular, these formulas are valid when all of the matrix entries are $\mathrm{i.i.d.}$ normal random variables. 

More recently, the generalized Lyapunov exponent has been used by Comtet, Texier, and Tourigny \cite{Texier, Comtet} to obtain results on the variance in the CLT for the product of $\mbox{SL}(2,\mathbb{R})$-valued random matrices of a particular form. Their work builds upon the transfer operator method used by Tutubalin \cite{Tutubalin} and is able provide explicit formulas for $\sigma^2$ in terms of certain special functions in the specific cases that they consider.

As discussed in Section \ref{sec:Lyapunov} in the context of the Lyapunov exponent, the singularity of the matrices \eqref{eq:1} poses a challenge for proving a CLT for their products since most of the existing theory is only relevant to $\mbox{GL}(d,\mathbb{R})$-valued matrices. Moreover, by considering the vectors $\mathbf{x}=(1,0)$ and $\mathbf{x}'=(0,1)$, it is clear that Cohen and Newman's result \eqref{eq:newman} also does not apply in this case since the distributions of $x_1$ and $1/x_1$ are different in general. So in order to establish the CLT and accompanying explicit formula for the variance which appears in Theorem \ref{thm:Main} below, we had to exploit a novel connection to the theory of $m$-dependent sequences which also leads to an interesting and precise nondegeneracy condition. 

The remainder of the paper is organized as follows. We state our main theorems in Section \ref{sec:main_results}. In Section \ref{sec:prelim}, we recall some aspects of the theory of $m$-dependent sequences and prove some preliminary results. The main theorems are proved in Section \ref{sec:proofs}. Finally, we conclude the paper in Section \ref{sec:examples} with several exact calculations of the Lyapunov exponent and variance in the CLT when the $x_j$ appearing in the matrices \eqref{eq:1} are binary, uniform, exponential, and Laplace random variables.

\section{Main results}\label{sec:main_results}

The main contributions of this paper, which are further detailed below, can be summarized as follows:
\begin{enumerate}
\item We fully generalize Proposition \ref{prop:ABL_Lyapunov} to all $\mathbb{R}\setminus\{0\}$-valued random variables satisfying an equivalent moment condition while simplifying the explicit formula for the Lyapunov exponent. This is done in Theorem \ref{thm:lyapunov_new}.
\item Under this general framework, we prove a CLT with an explicit formula for the variance for the product of the random singular matrices \eqref{eq:1} and also provide a precise nondegeneracy condition. %The proof involves a novel connection to the theory of $m$-dependent sequences involving classical and recent results. 
This is done in Theorem \ref{thm:Main}. 
\end{enumerate}

Our first main result is to generalize Proposition \ref{prop:ABL_Lyapunov} of Adams--Bloch--Lagarias to matrices of the form \eqref{eq:1} where the $x_j$ are $\mathbb{R}\setminus\{0\}$-valued random variables satisfying an equivalent moment condition. Unlike Proposition \ref{prop:ABL_Lyapunov}, we make no assumptions on the distribution of the $x_j$ beyond the moment condition \eqref{eq:two_moments}. In particular, our result is applicable regardless of whether the $x_j$ can be factored into a radial part and an independent random sign; see Remarks \ref{rem:ABL_implicit} and \ref{rem:lyapunov_new}.

\begin{thm}\label{thm:lyapunov_new}
Consider the random singular matrices of the form 
\[
Y_j=\left[\begin{array}{cc}
1 & x_j\\
\frac{1}{x_j} & 1
\end{array}\right],~j\geq 1
\]
where $\{x_j\}_{j\geq 1}$ is an i.i.d. sequence of $\mathbb{R}\setminus\{0\}$-valued random variables such that
\begin{equation}\label{eq:two_moments}
\mathbb{E}\left[\log\big(1+|x_1|\big)\right]<\infty~\text{ and }~\mathbb{E}\left[\log\left(1+\frac{1}{|x_1|}\right)\right]<\infty.
\end{equation}
Then almost surely
\[
\lambda=\lim_{n\to\infty}\frac{1}{n}\log\left\Vert S_n\right\Vert<\infty
\]
and the Lyapunov exponent has the explicit representation
\begin{equation}\label{eq:lyapunov_new}
\lambda=\mathbb{E}\Bigg[\log\left|1+\frac{x_2}{x_1}\right|\Bigg].
\end{equation}
\end{thm}
\begin{rem}\label{rem:lyapunov_new}
If the $x_j$ can be independently factored into $\mathrm{i.i.d.}$ radial parts $r_j>0$ and $\mathrm{i.i.d.}$ signs $\epsilon_j\in\{-1,1\}$, then the explicit representation \eqref{eq:lyapunov_new} is equivalent to that of \eqref{eq:lyapunov_old}. Hence Proposition \ref{prop:ABL_Lyapunov} can be seen as a special case of Theorem \ref{thm:lyapunov_new}.
\end{rem}
\begin{rem}\label{rem:Lyapunov_scale}
The Lyapunov exponent $\lambda$ is invariant under a scale transformation of the $x_j$; see also Remark \ref{rem:variance_scale} and the examples in Section \ref{sec:examples}.
\end{rem}

The second main result of this paper is the following CLT for the product of the random matrices featured in Theorem \ref{thm:lyapunov_new} and a formula for the variance $\sigma^2$. Besides the singularity of the matrices, the main novelty of this result is that the formula for $\sigma^2$ is given explicitly in terms of the distribution of the $x_j$ matrix entries. Moreover, aside from the moment condition \eqref{eq:moment_2}, this distribution can be completely general. As far as the authors know, this is the only nontrivial example where an explicit formula exists for the variance outside of the special cases discussed in Section \ref{sec:variance}. Additionally, in what may be a first appearance in the literature on products of random matrices, the proof uses both classical and recent results from the theory of $m$-dependent sequences to achieve its explicit variance formula and precise nondegeneracy condition. Perhaps these techniques could be useful in other singular matrix models for proving CLT's with explicit variance. 

\begin{thm}\label{thm:Main}
Suppose that in the setting of Theorem \ref{thm:lyapunov_new} we also have
\begin{equation}\label{eq:moment_2}
\mathbb{E}\left[\Bigg(\log\left|1+\frac{x_2}{x_1}\right|\Bigg)^2\right]<\infty.
\end{equation}
Then  
\[
\frac{1}{\sqrt{n}}\big(\log\left\| S_n\right\| -n\lambda\big)\stackrel{\mathcal{L}}{\to}N(0,\sigma^2)
\]
and the variance has the explicit representation 
\[
\sigma^2=\mathbb{E}\left[\Bigg(\log\left|1+\frac{x_2}{x_1}\right|\Bigg)^2\right]+2\,\mathbb{E}\Bigg[\log\left|1+\frac{x_2}{x_1}\right|\,\log\left|1+\frac{x_3}{x_2}\right|\Bigg]-3\lambda^2.
\]
Moreover, $\sigma^2=0$ if and only if the distribution of the $x_j$ takes one of the following three forms where $0<p<1$ and $a\neq 0$ are arbitrary:
\begin{enumerate}[label=\emph{\roman*.}]
\item $\displaystyle\mathbb{P}(x_j=a)=1$,
\item $\displaystyle\mathbb{P}(x_j=a)=p=1-\mathbb{P}\left(x_j=\big(-3-2\sqrt{2}\,\big)a\right)$,
\item $\displaystyle\mathbb{P}(x_j=a)=p=1-\mathbb{P}\left(x_j=\big(-3+2\sqrt{2}\,\big)a\right)$.
\end{enumerate}
\end{thm}
\begin{rem}\label{rem:degenerate}
While the existence of nondegenerate distributions for the $x_j$ such that $\sigma^2=0$ may seem surprising at first, we give an heuristic explanation for this fact in Section \ref{sec:binary}. Proving that $\mathrm{ii.}$ and $\mathrm{iii.}$ are the only such distributions is nontrivial and relies on a result of Janson \cite{Janson}; see the proof of Lemma \ref{lem:nondegenerate}.
\end{rem}
\begin{rem}\label{rem:variance_scale}
The variance $\sigma^2$ is invariant under a scale transformation of the $x_j$; see also Remark \ref{rem:Lyapunov_scale} and the examples in Section \ref{sec:examples}.
\end{rem}

%%%%%%%%%%%%%%%%%%%%%

\section{Preliminary results}\label{sec:prelim}

\subsection{On \texorpdfstring{$m$}{m}-dependent sequences and block factors}

The proof of Theorem \ref{thm:Main} uses a CLT for stationary sequences of $m$-dependent random variables due to Diananda \cite{Diananda}. For our purposes, a sequence of random variables $\{X_j\}_{j\geq 1}$ is called a \emph{stationary sequence} if it has the same law as the shifted sequence $\{X_{j+k}\}_{j\geq 1}$ for any shift $k\geq 1$. Moreover, we say that $\{X_j\}_{j\geq 1}$ is an $m$-\emph{dependent sequence} if $\{X_1,\dots,X_j\}$ is independent of $\{X_k,X_{k+1},\dots\}$ whenever $k-j>m$. 

Stationary $m$-dependent sequences often arise from $\mathrm{i.i.d.}$ sequences in the following manner. Let $\{x_j\}_{j\geq 1}$ be an $\mathrm{i.i.d.}$ sequence of random variables and let $f:\mathbb{R}^k\to\mathbb{R}$ be a measurable function. For each $j\geq 1$, define $X_j=f(x_j,\dots,x_{j+k-1})$. Then the sequence $\{X_j\}_{j\geq 1}$ is called a $k$-\emph{block factor} of the $\mathrm{i.i.d.}$ sequence $\{x_j\}_{j\geq 1}$. It is clear that any $k$-block factor of an $\mathrm{i.i.d.}$ sequence is a stationary $(k-1)$-dependent sequence, though the converse is not true in general; see \cite{2_block, blocks, recent_blocks}.

Unlike the classical CLT for $\mathrm{i.i.d.}$ sequences, the limit in the CLT for a stationary $m$-dependent sequence $\{X_j\}_{j\geq 1}$ can be degenerate with zero variance even if $\Var X_j>0$. However, Janson \cite{Janson} has completely characterized the particular case in which this can occur and we use his results to prove the precise nondegeneracy condition featured in Theorem \ref{thm:Main}.

\subsection{Product form}

The following lemma expresses the partial products of the matrices \eqref{eq:1} in a convenient form which reveals a hidden $2$-block factor structure.

\begin{lem}\label{lem:product_form}
Consider matrices of the form 
\[
Y_j=\left[\begin{array}{cc}
1 & x_j\\
\frac{1}{x_j} & 1
\end{array}\right],~j\geq 1
\]
where $x_j\in\mathbb{R}\setminus\{0\}$. Then for each $n\in\mathbb{N}$, the partial products $S_n=Y_n\cdots Y_2\cdot Y_1$ can be expressed as
\begin{equation}\label{eq:matrix_factor}
S_n=a_n\left[\begin{array}{cc}
1 & x_1 \\
\frac{1}{x_n} & \frac{x_1}{x_n}
\end{array}\right]
\end{equation}
where
\[
a_n=\prod_{j=1}^{n-1}\left(1+\frac{x_{j+1}}{x_j}\right).
\]
\end{lem}

\begin{proof}[Proof of Lemma \ref{lem:product_form}]
The representation \eqref{eq:matrix_factor} clearly holds for $n=1$. Supposing it holds for an arbitrary $n\in\mathbb{N}$, we have
\begin{align*}
S_{n+1}=Y_{n+1}S_n &=a_n\left[\begin{array}{cc}
1 & x_{n+1}\\
\frac{1}{x_{n+1}} & 1
\end{array}\right]\left[\begin{array}{cc}
1 & x_1 \\
\frac{1}{x_n} & \frac{x_1}{x_n}
\end{array}\right]\\
&=a_n\left[\begin{array}{cc}
1+\frac{x_{n+1}}{x_n} & x_1\left(1+\frac{x_{n+1}}{x_n}\right)\\
\frac{1}{x_{n+1}}\left(1+\frac{x_{n+1}}{x_n}\right) & \frac{x_1}{x_{n+1}}\left(1+\frac{x_{n+1}}{x_n}\right)
\end{array}\right]\\
&=a_{n+1}\left[\begin{array}{cc}
1 & x_1\\
\frac{1}{x_{n+1}} & \frac{x_1}{x_{n+1}}
\end{array}\right].
\end{align*}
Hence the desired result follows by induction.
\end{proof}

Taking the $\log$ of the norm of \eqref{eq:matrix_factor} leads to the identity
\begin{equation}\label{eq:product_block}
\log\|S_n\|=\sum_{j=1}^{n-1}\log\left|1+\frac{x_{j+1}}{x_j}\right|+\log\left\|\left[\begin{array}{cc}
1 & x_1 \\
\frac{1}{x_n} & \frac{x_1}{x_n}
\end{array}\right]\right\|,
\end{equation}
where the partial sum of a $2$-block factor of the $\mathrm{i.i.d.}$ sequence $\{x_j\}_{j\geq 1}$ becomes obvious. This representation will be exploited later in the proof of Theorem \ref{thm:Main}.

\subsection{Binary distributed \texorpdfstring{$x_j$}{x j}}\label{sec:binary}

In this section we evaluate the Lyapunov exponent and variance formulas from Theorems \ref{thm:lyapunov_new} and \ref{thm:Main} when the $x_j$ are general binary random variables. Note that the moment condition \eqref{eq:moment_2} implies that the two values $a,b\neq 0$ that the $x_j$ take must satisfy $a\neq -b$. While the computations are tedious and relatively straightforward, we give a complete proof since this result is essential for proving the nondegeneracy condition of Theorem \ref{thm:Main} and also appears as Example \ref{ex:binary}.

\begin{lem}\label{lem:binary}
Suppose the $x_j$ in Theorems \ref{thm:lyapunov_new} and \ref{thm:Main} are binary random variables taking the two values $a,b\neq 0$ with $a\neq -b$ and $\mathbb{P}\left(x_j=a\right)=p=1-\mathbb{P}\left(x_j=b\right)$. Then we have
\begin{align}
&\mathbb{E}\left[\log\left|1+\frac{x_2}{x_1}\right|\right]=\log 2+p(1-p)\log\left(\frac{(a+b)^2}{4|ab|}\right),\label{eq:bin_lambda}\\
\nonumber\\
&\mathbb{E}\left[\left(\log\left|1+\frac{x_2}{x_1}\right|\right)^2\right]+2\,\mathbb{E}\left[\log\left|1+\frac{x_2}{x_1}\right|\,\log\left|1+\frac{x_3}{x_2}\right|\right]-3\,\mathbb{E}\left[\log\left|1+\frac{x_2}{x_1}\right|\right]^2\nonumber\\
&=p(1-p)\big(1-3p(1-p)\big)\left(\log\left(\frac{(a+b)^2}{4|ab|}\right)\right)^2.\label{eq:bin_sigma}
\end{align}
\end{lem}

\begin{proof}[Proof of Lemma \ref{lem:binary}]
We can prove \eqref{eq:bin_lambda} by writing 
\begin{align*}
&\mathbb{E}\left[\log\left|1+\frac{x_2}{x_1}\right|\right]\\
&=\left(p^2+(1-p)^2\right)\log 2+p(1-p)\left(\log\left|1+\frac{a}{b}\right|+\log\left|1+\frac{b}{a}\right|\right)\\
&=\left(p^2+(1-p)^2\right)\log 2+p(1-p)\log\left(\frac{(a+b)^2}{|ab|}\right)\\
&=\left(p^2+(1-p)^2\right)\log 2+p(1-p)\left(\log 4+\log\left(\frac{(a+b)^2}{4|ab|}\right)\right)\\
&=\log 2+\underbrace{p(1-p)\log\left(\frac{(a+b)^2}{4|ab|}\right)}_{\displaystyle A}.
\end{align*}
Here we have labeled the $A$ term for later use in proving \eqref{eq:bin_sigma}. Note that the trick of extracting a $\log 4$ term will be used twice more in what follows.

Next we compute the first two terms that appear on the left-hand side of \eqref{eq:bin_sigma} by writing
\begin{align*}
&\mathbb{E}\left[\left(\log\left|1+\frac{x_2}{x_1}\right|\right)^2\right]\\
&=\underbrace{\left(p^2+(1-p)^2\right)(\log 2)^2}_{\displaystyle B}+\underbrace{p(1-p)\left(\left(\log\left|1+\frac{a}{b}\right|\right)^2+\left(\log\left|1+\frac{b}{a}\right|\right)^2\right)}_{\displaystyle C}
\end{align*}
and
\begin{align*}
&\mathbb{E}\left[\log\left|1+\frac{x_2}{x_1}\right|\,\log\left|1+\frac{x_3}{x_2}\right|\right]\\
&=\left(p^3+(1-p)^3\right)(\log 2)^2\\
&~~+\left(p^2(1-p)+p(1-p)^2\right)\log 2\left(\log\left|1+\frac{a}{b}\right|+\log\left|1+\frac{b}{a}\right|\right)\\
&~~+\left(p^2(1-p)+p(1-p)^2\right)\log\left|1+\frac{a}{b}\right|\,\log\left|1+\frac{b}{a}\right|\\
&=\underbrace{\big(1-p(1-p)\big)(\log 2)^2}_{\displaystyle D}+\underbrace{p(1-p)\log 2\,\log\left(\frac{(a+b)^2}{4|ab|}\right)}_{\displaystyle E}\\
&~~+\underbrace{p(1-p)\log\left|1+\frac{a}{b}\right|\,\log\left|1+\frac{b}{a}\right|}_{\displaystyle F}.
\end{align*}
The $C$ and $F$ terms can be combined into a perfect square and then expanded using the $\log 4$ trick from before. More precisely, we have 
\begin{align*}
C+2F&=p(1-p)\left(\log\left|1+\frac{a}{b}\right|+\log\left|1+\frac{b}{a}\right|\right)^2\\
&=\underbrace{4p(1-p)(\log 2)^2}_{\displaystyle G}+\underbrace{4p(1-p)\log 2\,\log\left(\frac{(a+b)^2}{4|ab|}\right)}_{\displaystyle H}\\
&~~+\underbrace{p(1-p)\left(\log\left(\frac{(a+b)^2}{4|ab|}\right)\right)^2}_{\displaystyle I}.
\end{align*}

Finally, we combine these expressions for the first two terms that appear on the left-hand side of \eqref{eq:bin_sigma} with the square of the right-hand side of \eqref{eq:bin_lambda} to yield
\begin{align*}
&\mathbb{E}\left[\left(\log\left|1+\frac{x_2}{x_1}\right|\right)^2\right]+2\,\mathbb{E}\left[\log\left|1+\frac{x_2}{x_1}\right|\,\log\left|1+\frac{x_3}{x_2}\right|\right]-3\,\mathbb{E}\left[\log\left|1+\frac{x_2}{x_1}\right|\right]^2\\
&=(B+C)+2(D+E+F)-3\left((\log 2)^2+2A\log 2+A^2\right)\\
&=\left(B+2D-3(\log 2)^2\right)+(C+2F)+\left(2E-6A\log 2\right)-3A^2\\
&=\left(B+2D-3(\log 2)^2+G\right)+\left(2E-6A\log 2+H\right)+\left(I-3A^2\right)\\
&=0+0+\left(p(1-p)-3p^2(1-p)^2\right)\left(\log\left(\frac{(a+b)^2}{4|ab|}\right)\right)^2.
\end{align*}
\end{proof}

Since \eqref{eq:bin_sigma} expresses the variance $\sigma^2$ in Theorem \ref{thm:Main}, any admissible pair of atoms $a$ and $b$ which solves $(a+b)^2=4|ab|$ will correspond to a family of binary distributions indexed by $0<p<1$ that has $\sigma^2=0$. In particular, this leads to the nondegenerate distributions of the form $\mathrm{ii.}$ and $\mathrm{iii.}$ which have zero variance in the CLT of Theorem \ref{thm:Main}; see also Remark \ref{rem:degenerate}. While the existence of such distributions may seem surprising when compared to the classical CLT, we give a nonrigorous heuristic explanation below.

First note that from Lemma \ref{lem:product_form} and its consequence \eqref{eq:product_block}, we have
\[
\log\|S_n\|=\sum_{j=1}^{n-1}\log\left|1+\frac{x_{j+1}}{x_j}\right|+O_p(1)\text{ as }n\to\infty.
\]
Now suppose that the $x_j$ have a binary distribution of the form $\mathrm{ii.}$ or $\mathrm{iii.}$ and define $A_j=\log|1+\frac{x_{j+1}}{x_j}|-\log 2$. Notice that each $A_j$ term can take one of three values: $0$, $\log(1+\sqrt{2})$, or $-\log(1+\sqrt{2})$. Moreover, the last two values will tend to occur in equal proportion and if the $A_j$ were independent, then the deviation of the sum from $0$ would be of order $\sqrt{n}$ and we would have a nondegenerate limit in the CLT. However, the particular $1$-dependent structure of the $A_j$ leads to a significant reduction in this deviation.

To investigate this reduced deviation, it will be instructive to split the sum of $2n$ of the $A_j$ into even and odd parts. Notice that the even part $\sum_{j=1}^n A_{2j}$ is a sum of $n$ $\mathrm{i.i.d.}$ terms with mean $0$ and finite variance so its deviation from $0$ is of order $\sqrt{n}$ and the same can be said for the odd part. However, the deviations have a high negative correlation so they tend to cancel out in the full sum. To see where this negative correlation comes from, notice that if $A_j=\log(1+\sqrt{2})$, then neither $A_{j-1}$ nor $A_{j+1}$ can equal $\log(1+\sqrt{2})$, though they both can equal $-\log(1+\sqrt{2})$. Likewise, if $A_j=-\log(1+\sqrt{2})$, then neither $A_{j-1}$ nor $A_{j+1}$ can equal $-\log(1+\sqrt{2})$, though they both can equal $\log(1+\sqrt{2})$. Hence the even and odd parts tend to have deviations of opposite sign and the resulting cancellation leads to 
\[
\log\|S_n\|=n\log2 +o_p(\sqrt{n})\text{ as }n\to\infty.
\]
Now it follows that $\frac{1}{\sqrt{n}}(\log\|S_n\|-n\log 2)$ has a degenerate limit.

\section{Proofs of the main results}\label{sec:proofs}

\subsection{Proof of Theorem \ref{thm:lyapunov_new}}

\begin{proof}[Proof of Theorem \ref{thm:lyapunov_new}]
Let $\|\cdot\|$ denote the Hilbert-Schmidt norm. Then we have 
\begin{align*}
\left\Vert Y_1\right\Vert&=\sqrt{x_1^2+2+\frac{1}{x_1^2}}\\
&=\sqrt{\left(1+x_1^2\right)\left(1+\frac{1}{x_1^2}\right)}\\
&\geq 1.
\end{align*}
Along with the moment condition \eqref{eq:two_moments} and subadditivity of the square root, this implies that
\begin{align*}
\mathbb{E}\left[\log^+ \left\Vert Y_1\right\Vert\right]&=\mathbb{E}\big[\log\left\Vert Y_1\right\Vert\big]\\
&=\frac{1}{2}\Bigg(\mathbb{E}\left[\log\left(1+x_1^2\right)\right]+\mathbb{E}\left[\log\left(1+\frac{1}{x_1^2}\right)\right]\Bigg)\\
&\leq\mathbb{E}\left[\log\big(1+|x_1|\big)\right]+\mathbb{E}\left[\log\left(1+\frac{1}{|x_1|}\right)\right]\\
&<\infty.
\end{align*}
Hence by \cite[Theorem 2]{FurKest}, almost surely we have
\[
\lambda=\lim_{n\to\infty}\frac{1}{n}\log\left\Vert S_{n}\right\Vert<\infty.
\]

Next we compute $\lambda$. Arguing similarly using \eqref{eq:two_moments}, we know that
\begin{equation}\label{eq:remainder}
\begin{split}
\mathbb{E}\left[\Bigg|\log\left\|\left[\begin{array}{cc}
1 & x_1 \\
\frac{1}{x_n} & \frac{x_1}{x_n}
\end{array}\right]\right\|\Bigg|\right]&=\frac{1}{2}\Bigg(\mathbb{E}\left[\log\left(1+x_1^2\right)\right]+\mathbb{E}\left[\log\left(1+\frac{1}{x_n^2}\right)\right]\Bigg)\\
%&\leq\mathbb{E}\left[\log\big(1+|x_1|\big)\right]+\mathbb{E}\left[\log\left(1+\frac{1}{|x_1|}\right)\right]\\
&<\infty.
\end{split}
\end{equation}
Hence for each $n\geq 1$, we can use Lemma \ref{lem:product_form} and \eqref{eq:product_block} in particular to write
\begin{equation}\label{eq:split}
\mathbb{E}\big[\log\left\| S_n\right\|\big]=\mathbb{E}\left[\sum_{j=1}^{n-1}\log\left|1+\frac{x_{j+1}}{x_j}\right|\right]+\mathbb{E}\Bigg[\log\left\|\left[\begin{array}{cc}
1 & x_1 \\
\frac{1}{x_n} & \frac{x_1}{x_n}
\end{array}\right]\right\|\Bigg].
\end{equation}
Moreover, since 
\begin{align*}
\left|1+\frac{x_2}{x_1}\right|&\leq 1+\frac{|x_2|}{|x_1|}\\
&\leq \left(1+\frac{1}{|x_1|}\right)\Big(1+|x_2|\Big),
\end{align*}
the moment condition \eqref{eq:two_moments} implies
\begin{equation}\label{eq:product}
-\infty\leq\mathbb{E}\Bigg[\log\left|1+\frac{x_2}{x_1}\right|\Bigg]<\infty.
\end{equation}
Now it follows from the definition of $\lambda$ along with \eqref{eq:remainder}, \eqref{eq:split}, and \eqref{eq:product} that
\begin{align*}
\lambda&=\lim_{n\to\infty}\frac{1}{n}\left(\mathbb{E}\left[\sum_{j=1}^{n-1}\log\left|1+\frac{x_{j+1}}{x_j}\right|\right]+\mathbb{E}\Bigg[\log\left\|\left[\begin{array}{cc}
1 & x_1 \\
\frac{1}{x_n} & \frac{x_1}{x_n}
\end{array}\right]\right\|\Bigg]\right)\\
&=\mathbb{E}\Bigg[\log\left|1+\frac{x_2}{x_1}\right|\Bigg].
\end{align*}
\end{proof}

\subsection{Proof of Theorem \ref{thm:Main}}

\begin{proof}[Proof of Theorem \ref{thm:Main}]
First note that by Theorem \ref{thm:lyapunov_new}, the moment condition \eqref{eq:moment_2} implies that $|\lambda|<\infty$. Hence for each $n\geq 1$, we can use \eqref{eq:product_block} to write
\begin{equation}\label{eq:CLT_split}
\log\left\| S_n\right\| -n\lambda=\sum_{j=1}^{n-1}\underbrace{\left(\log\left|1+\frac{x_{j+1}}{x_j}\right|-\lambda\right)}_{\displaystyle A_j}+\underbrace{\log\left\|\left[\begin{array}{cc}
1 & x_1 \\
\frac{1}{x_n} & \frac{x_1}{x_n}
\end{array}\right]\right\|-\lambda}_{\displaystyle B_n}.
\end{equation}

It is clear that $\{A_j\}_{j\geq 1}$ is a $2$-block factor of the $\mathrm{i.i.d.}$ sequence $\{x_j\}_{j\geq 1}$, hence it is also a stationary $1$-dependent sequence. Additionally, we have 
\begin{equation}\label{eq:A_mean}
\mathbb{E}[A_j]=0
\end{equation}
and 
\begin{equation}\label{eq:A_variance}
\mathbb{E}\left[A_j^2\right]=\mathbb{E}\left[\Bigg(\log\left|1+\frac{x_2}{x_1}\right|\Bigg)^2\right]-\lambda^2.
\end{equation}

Define $C_{j-k}=\mathbb{E}[A_j A_k]$ for $j,k\geq 1$. Then $C_0$ is given by \eqref{eq:A_variance}, while $1$-dependence and \eqref{eq:A_mean} imply that $C_j=0$ whenever $|j|>1$. Furthermore, we have
\[
C_1=C_{-1}=\mathbb{E}\Bigg[\log\left|1+\frac{x_2}{x_1}\right|\,\log\left|1+\frac{x_3}{x_2}\right|\Bigg]-\lambda^2.
\]
Now Diananda's CLT for stationary $m$-dependent sequences from \cite[Theorem 2]{Diananda} implies that 
\[
\frac{1}{\sqrt{n}}\sum_{j=1}^n A_j\stackrel{\mathcal{L}}{\to}N(0,\sigma^2)
\]
where 
\begin{align}
\sigma^2&=\sum_{j=-1}^1 C_j\nonumber\\
&=\mathbb{E}\left[\Bigg(\log\left|1+\frac{x_2}{x_1}\right|\Bigg)^2\right]+2\,\mathbb{E}\Bigg[\log\left|1+\frac{x_2}{x_1}\right|\,\log\left|1+\frac{x_3}{x_2}\right|\Bigg]-3\lambda^2.\label{eq:variance}
\end{align}
It follows from the Cauchy--Schwarz inequality and the moment condition \eqref{eq:moment_2} that $\sigma^2<\infty$.

Turning our attention back to \eqref{eq:CLT_split}, an application of Markov's inequality along with \eqref{eq:remainder} shows that $\frac{1}{\sqrt{n}}B_n\to 0$ in probability. Now Slutsky's theorem \cite[Theorem 11.4]{gut_book} can be used to conclude that 
\[
\frac{1}{\sqrt{n}}\big(\log\left\| S_n\right\| -n\lambda\big)\stackrel{\mathcal{L}}{\to}N(0,\sigma^2)
\]
where $\sigma^2$ is given by \eqref{eq:variance}.

It remains to show that $\sigma^2=0$ if and only if the $x_j$ are distributed according to one of the three forms listed in the statement of Theorem \ref{thm:Main}. This is addressed separately in Lemma \ref{lem:nondegenerate} below. 
\end{proof}

\begin{lem}\label{lem:nondegenerate}
The variance $\sigma^2$ in Theorem \ref{thm:Main} equals zero if and only if the distribution of the $x_j$ takes one of the following three forms where $0<p<1$ and $a\neq 0$ are arbitrary:
\begin{enumerate}[label=\emph{\roman*.}]
\item $\displaystyle\mathbb{P}(x_j=a)=1$,
\item $\displaystyle\mathbb{P}(x_j=a)=p=1-\mathbb{P}\left(x_j=\big(-3-2\sqrt{2}\,\big)a\right)$,
\item $\displaystyle\mathbb{P}(x_j=a)=p=1-\mathbb{P}\left(x_j=\big(-3+2\sqrt{2}\,\big)a\right)$.
\end{enumerate}
\end{lem}

\begin{proof}[Proof of Lemma \ref{lem:nondegenerate}]
We know from Lemma \ref{lem:binary} that $\sigma^2=0$ if the $x_j$ are distributed according to any of these three distributions so we can focus on showing that these are the only possible distributions such that $\sigma^2=0$. Accordingly, we now assume that $\sigma^2=0$ and deduce what restrictions this imposes on the distribution of the $x_j$. 

From the proof of Theorem \ref{thm:Main}, we know that $\sigma^2$ is the variance in the CLT for a stationary $1$-dependent sequence that is also a $2$-block factor of an $\mathrm{i.i.d.}$ sequence. Hence if $\sigma^2=0$, we can use \cite[Corollary 2]{Janson} to deduce that there exists a measurable function $g:\mathbb{R}\to\mathbb{R}$ such that for all $n\geq1$,
\begin{equation}\label{JanEquation:g}
\sum_{j=1}^n A_j=g\left(x_{n+1}\right)-g\left(x_1\right)~\text{almost surely}
\end{equation}
with $g$ being almost surely unique up to an additive constant and where $A_j$ was defined in \eqref{eq:CLT_split}. Using this fact, we treat separately the case where the distribution of the $x_j$ has no atoms and the case where it has at least one atom.\\ 

\noindent\textbf{Case 1: no atoms}

First we suppose that the distribution of the $x_j$ is nonatomic and get a contradiction. Using $n=2$ in \eqref{JanEquation:g}, we have
\[
\log\left|\left(1+\frac{x_2}{x_1}\right)\left(1+\frac{x_3}{x_2}\right)\right|-2\lambda=g\left(x_3\right)-g\left(x_1\right)~\text{almost surely},
\]
from which it follows that  
\begin{equation}\label{eq:Jan2-1b}
\left(x_1+x_2\right)^2\left(x_2+x_3\right)^2=x_1^2\, x_2^2\,e^{4\lambda+2g\left(x_3\right)-2g\left(x_1\right)}~\text{almost surely}.
\end{equation}
Expanding \eqref{eq:Jan2-1b} leads to a quartic equation in $x_2$ whose coefficients are functions of $x_1$ and $x_3$. Denote the four solutions of this equation that are given by the quartic formula by $h_i(x_1,x_3)$, $i=1,2,3,4$. Then \eqref{eq:Jan2-1b} implies that 
\begin{equation}\label{eq:4_as}
\mathbb{P}\left(\bigcup_{i=1}^4 \big\{x_2=h_i(x_1,x_3)\big\}\right)=1.
\end{equation}
By \eqref{eq:4_as}, we know that $\mathbb{P}(x_2-h_i(x_1,x_3)=0)>0$ for some $i$. However, since $x_2$ and $h_i(x_1,x_3)$ are independent, this contradicts the fact that the distribution of $x_2$ has no atoms. It follows that if $\sigma^2=0$, then the distribution of the $x_j$ must have at least one atom.\\

\noindent\textbf{Case 2: at least one atom}

Next we suppose that the distribution of the $x_j$ has an atom at $a\neq 0$. Using $n=1$ in \eqref{JanEquation:g}, we have
\begin{equation}\label{eq:Jan1}
\log\left|1+\frac{x_2}{x_1}\right|-\lambda=g\left(x_2\right)-g\left(x_1\right)~\text{almost surely}.
\end{equation}
By independence we have $\mathbb{P}\left(x_{1}=x_{2}=a\right)>0$, so it follows from \eqref{eq:Jan1} that $\mathbb{P}\left(\lambda=\log2\right)>0$ which implies that 
\begin{equation}\label{eq:as_lambda}
\lambda=\log2. 
\end{equation}

Using $n=2$ in \eqref{JanEquation:g} along with \eqref{eq:as_lambda} leads to
\begin{equation}\label{eq:Jan2}
\log\left|\left(1+\frac{x_2}{x_1}\right)\left(1+\frac{x_3}{x_2}\right)\right|-2\log2=g\left(x_3\right)-g\left(x_1\right)~\text{almost surely}.
\end{equation}
Supposing that $\mathbb{P}\left(x_1=a\right)=\alpha>0$, we can use independence and \eqref{eq:Jan2} to write 
\begin{align*}
\alpha^2 & =\mathbb{P}\left(x_1=a,\,x_3=a\right)\\
 & =\mathbb{P}\left(x_1=a,\,x_3=a,\,\left|\left(1+\frac{x_2}{x_1}\right)\left(1+\frac{x_3}{x_2}\right)\right|=4\,e^{g\left(x_3\right)-g\left(x_1\right)}\right)\\
%&=\mathbb{P}\left(x_1=a,\,x_3=a,\,\log\left|1+\frac{x_2}{a}\right|+\log\left|1+\frac{a}{x_2}\right|-2\log2=g\left(a\right)-g\left(a\right)\right)\\
 & =\mathbb{P}\left(x_1=a,\,x_3=a,\,\left|2+\frac{x_2}{a}+\frac{a}{x_2}\right|=4\right)\\
 & =\alpha^2\,\mathbb{P}\Bigg(\left|2+\frac{x_2}{a}+\frac{a}{x_2}\right|=4\Bigg).
\end{align*}
Since $\alpha>0$, this implies that 
\begin{equation}\label{eq:forced}
\mathbb{P}\Bigg(\left|2+\frac{x_2}{a}+\frac{a}{x_2}\right|=4\Bigg)=1.
\end{equation}
The three solutions of 
\[
\left|2+\frac{x_2}{a}+\frac{a}{x_2}\right|=4
\]
can be found with straightforward calculations, hence \eqref{eq:forced} forces the $x_j$ to be discrete random variables taking values in the set
\begin{equation}\label{eq:atoms}
\left\{a,\,\left(-3-2\sqrt{2}\right)a,\,\left(-3+2\sqrt{2}\right)a\right\}.
\end{equation}
It remains to show that of all the probability distributions which are supported on subsets of \eqref{eq:atoms}, the only ones such that $\sigma^2=0$ are of the forms $\mathrm{i.}$--$\mathrm{iii.}$

The set \eqref{eq:atoms} has seven nonempty subsets, hence probability distributions created solely from atoms drawn from \eqref{eq:atoms} can have seven different supports. Any distribution whose support is one of the three singleton subsets is of the form $\mathrm{i.}$ while any distribution whose support is one of the two doubleton subsets containing $a$ is of the form $\mathrm{ii.}$ or $\mathrm{iii.}$ The remaining two subsets both contain the points $(-3-2\sqrt{2})a$ and $(-3+2\sqrt{2})a$. Hence the proof is complete if we can show that a distribution which assigns positive probability to each of these two points must have $\sigma^2\neq 0$, or equivalently, show that $\sigma^2=0$ implies that at most one of those two points is assigned positive probability.

Towards this end, define $a_-$ and $a_+$ by
\[
a_\pm=\left(-3\pm 2\sqrt{2}\right)a
\]
and define $\alpha_+$ and $\alpha_-$ by
\[
\alpha_\pm=\mathbb{P}\left(x_j=a_\pm\right).
\]
Then assuming $\sigma^2=0$, we can use independence and \eqref{eq:Jan2} to write
\begin{align*}
\alpha_+^2\,\alpha_- &=\mathbb{P}\left(x_1=x_3=a_+,\,x_2=a_-\right)\\
&=\mathbb{P}\left(x_1=x_3=a_+,\,x_2=a_-,\,\left|\left(1+\frac{x_2}{x_1}\right)\left(1+\frac{x_3}{x_2}\right)\right|=4\,e^{g(x_3)-g(x_1)}\right)\\
&=\mathbb{P}\left(x_1=x_3=a_+,\,x_2=a_-,\,\left|\left(18+12\sqrt{2} \right) \left(18-12\sqrt{2} \right)\right|=4\right)\\
&=\mathbb{P}\left(x_1=x_3=a_+,\,x_2=a_-,\,36=4\right)\\
&=0.
\end{align*}
Hence at least one of $\alpha_+=0$ or $\alpha_-=0$ holds. It follows that the only distributions such that $\sigma^2=0$ are of the forms $\mathrm{i.}$--$\mathrm{iii.}$
\end{proof}

\section{Examples}\label{sec:examples}

In this section we compute the Lyapunov exponent and variance in the CLT for the product of random matrices of the form \eqref{eq:1} for some specific examples using the formulas in Theorems \ref{thm:lyapunov_new} and \ref{thm:Main}. Example \ref{ex:binary} is taken directly from Lemma \ref{lem:binary}. For the other examples, the calculations were carried out by first computing the joint distribution of $\frac{x_2}{x_1}$ and $\frac{x_3}{x_2}$. The computations were subsequently checked by both symbolic and numerical integration using Mathematica and further corroborated by Monte Carlo simulations. We leave the details up to the reader. As remarked upon earlier, these examples also serve to demonstrate how the Lyapunov exponent and variance are invariant under a scale transformation of the $x_j$.

\begin{example}\label{ex:binary}
\emph{Binary distribution}\\
Suppose the $x_j$ are binary random variables taking the two values $a,b\neq 0$ with $a\neq -b$ and $\mathbb{P}\left(x_j=a\right)=p=1-\mathbb{P}\left(x_j=b\right)$. Then the Lyapunov exponent is given by 
\[
\lambda = \log 2+p(1-p)\log\left(\frac{(a+b)^2}{4|ab|}\right)
\]
and the variance in the CLT is 
\[
\sigma^2=p(1-p)\big(1-3p(1-p)\big)\Bigg(\log\left(\frac{(a+b)^2}{4|ab|}\right)\Bigg)^2.
\]
\end{example}

\begin{example}\label{ex:uniform}
\emph{Uniform distribution}\\
Suppose the $x_j$ are uniformly distributed over the interval $[-a,b]$ where we require $-a\leq 0<b$. Then the Lyapunov exponent is given by
\[
\lambda=
\begin{cases}
2\log 2-\frac{1}{2}\approx 0.8863&\text{if }a=0\\
\log 2-\frac{1}{2}\approx 0.1931&\text{if }a=b\\
2\frac{a^2+b^2}{(a+b)^2}\log 2-\frac{1}{2}+\frac{a-b}{(a+b)^2}\Big(b\log\left|1-\frac{a}{b}\right|-a\log\left|1-\frac{b}{a}\right|\Big)&\text{otherwise.}
\end{cases}
\]
In the particular cases of $a=0$ and $a=b$, we can also exactly calculate the variance in the CLT as 
\[
\sigma^2=
\begin{cases}
\frac{1}{36}\left(4\pi^2+15\right)-\frac{2}{3}\log 2\left(7\log 2-2\right)\approx 0.1954&\text{if }a=0\\
\frac{1}{36}\left(5\pi^2+15\right)\approx 1.7874&\text{if }a=b.
\end{cases}
\]
\end{example}

\begin{example}
\emph{Exponential distribution}\\
Suppose the $x_j$ have a nondegenerate exponential distribution. Then the Lyapunov exponent is given by
\[
\lambda=1
\]
and the variance in the CLT is 
\[
\sigma^2=\frac{1}{3}\left(\pi^2-9\right)\approx 0.2899.
\]
\end{example}

\begin{example}
\emph{Laplace distribution}\\
Suppose the $x_j$ have a nondegenerate Laplace distribution (symmetric bilateral exponential distribution) with zero mean. Then the Lyapunov exponent is given by
\[
\lambda=\frac{1}{2}
\]
and the variance in the CLT is 
\[
\sigma^2=\frac{1}{36}\left(8\pi^2-27\right)\approx 1.4432.
\]
\end{example}

\begin{acknowledgments}
The authors would like to thank Alexander Teplyaev and Maria Gordina for helpful discussions. We would also like to thank two anonymous referees for their careful review of the manuscript and whose comments and suggestions led to an improved paper. 
\end{acknowledgments}

\newcommand{\etalchar}[1]{$^{#1}$}

\end{document}